\documentclass[11pt]{article}
\usepackage[margin=1in]{geometry} 
\usepackage{amssymb,amsfonts,amsmath,bbm,mathrsfs,stmaryrd,mathtools,mathabx}
\usepackage{xcolor}
\usepackage{url}

\usepackage{extarrows}

\usepackage[shortlabels]{enumitem}
\usepackage{tensor}

\usepackage[T1]{fontenc}

\usepackage[colorlinks,
linkcolor=black!75!red,
citecolor=blue,
pdftitle={},
pdfproducer={pdfLaTeX},
pdfpagemode=None,
bookmarksopen=true,
bookmarksnumbered=true,
backref=page]{hyperref}

\usepackage{tikz}
\usetikzlibrary{arrows,calc,decorations.pathreplacing,decorations.markings,intersections,shapes.geometric,through,fit,shapes.symbols,positioning,decorations.pathmorphing}

\usepackage{braket}

\usepackage[amsmath,thmmarks,hyperref]{ntheorem}
\usepackage{cleveref}

\creflabelformat{enumi}{#2#1#3}

\crefname{section}{Section}{Sections}
\crefformat{section}{#2Section~#1#3} 
\Crefformat{section}{#2Section~#1#3} 

\crefname{subsection}{\S}{\S\S}
\AtBeginDocument{%
  \crefformat{subsection}{#2\S#1#3}%
  \Crefformat{subsection}{#2\S#1#3}%
}

\crefname{subsubsection}{\S}{\S\S}
\AtBeginDocument{%
  \crefformat{subsubsection}{#2\S#1#3}%
  \Crefformat{subsubsection}{#2\S#1#3}%
}

\theoremstyle{plain}

\newtheorem{lemma}{Lemma}[section]

\newtheorem{corollary}[lemma]{Corollary}
\newtheorem{theorem}[lemma]{Theorem}

\theoremstyle{plain}
\theoremnumbering{Alph}
\newtheorem{theoremN}{Theorem}

\theoremstyle{plain}
\theorembodyfont{\upshape}
\theoremsymbol{\ensuremath{\blacklozenge}}

\newtheorem{definition}[lemma]{Definition}

\newtheorem{examples}[lemma]{Examples}
\newtheorem{remark}[lemma]{Remark}
\newtheorem{remarks}[lemma]{Remarks}

\crefname{definition}{definition}{definitions}
\crefformat{definition}{#2definition~#1#3} 
\Crefformat{definition}{#2Definition~#1#3} 

\crefname{example}{example}{examples}
\crefformat{example}{#2example~#1#3} 
\Crefformat{example}{#2Example~#1#3} 

\crefname{examples}{example}{examples}
\crefformat{examples}{#2example~#1#3} 
\Crefformat{examples}{#2Example~#1#3} 

\crefname{remark}{remark}{remarks}
\crefformat{remark}{#2remark~#1#3} 
\Crefformat{remark}{#2Remark~#1#3} 

\crefname{remarks}{remark}{remarks}
\crefformat{remarks}{#2remark~#1#3} 
\Crefformat{remarks}{#2Remark~#1#3} 

\crefname{convention}{convention}{conventions}
\crefformat{convention}{#2convention~#1#3} 
\Crefformat{convention}{#2Convention~#1#3} 

\crefname{notation}{notation}{notations}
\crefformat{notation}{#2notation~#1#3} 
\Crefformat{notation}{#2Notation~#1#3} 

\crefname{table}{table}{tables}
\crefformat{table}{#2table~#1#3} 
\Crefformat{table}{#2Table~#1#3}

\crefname{lemma}{lemma}{lemmas}
\crefformat{lemma}{#2lemma~#1#3} 
\Crefformat{lemma}{#2Lemma~#1#3} 

\crefname{proposition}{proposition}{propositions}
\crefformat{proposition}{#2proposition~#1#3} 
\Crefformat{proposition}{#2Proposition~#1#3} 

\crefname{corollary}{corollary}{corollaries}
\crefformat{corollary}{#2corollary~#1#3} 
\Crefformat{corollary}{#2Corollary~#1#3} 

\crefname{theorem}{theorem}{theorems}
\crefformat{theorem}{#2theorem~#1#3} 
\Crefformat{theorem}{#2Theorem~#1#3} 

\crefname{enumi}{}{}
\crefformat{enumi}{#2#1#3}
\Crefformat{enumi}{#2#1#3}

\crefname{assumption}{assumption}{Assumptions}
\crefformat{assumption}{#2assumption~#1#3} 
\Crefformat{assumption}{#2Assumption~#1#3} 

\crefname{construction}{construction}{Constructions}
\crefformat{construction}{#2construction~#1#3} 
\Crefformat{construction}{#2Construction~#1#3} 

\crefname{equation}{}{}
\crefformat{equation}{(#2#1#3)} 
\Crefformat{equation}{(#2#1#3)}

\numberwithin{equation}{section}

\theoremstyle{nonumberplain}
\theoremsymbol{\ensuremath{\blacksquare}}

\newtheorem{proof}{Proof}
\newcommand\pf[1]{\newtheorem{#1}{Proof of \Cref{#1}}}

\newcommand\bC{{\mathbb C}}

\newcommand\bG{{\mathbb G}}

\newcommand\bZ{{\mathbb Z}}

\newcommand\cA{{\mathcal A}}

\newcommand\cK{{\mathcal K}}
\newcommand\cL{{\mathcal L}}

\newcommand\cP{{\mathcal P}}

\DeclareMathOperator{\im}{\mathrm{im}}
\DeclareMathOperator{\diag}{diag}
\DeclareMathOperator{\diam}{diam}
\DeclareMathOperator{\Irr}{Irr}

\newcommand{\cat}[1]{\textsc{#1}}

\newcommand{\qedhere}{\mbox{}\hfill\ensuremath{\blacksquare}}

\title{Normal approximations of commuting square-summable matrix families}
\author{Alexandru Chirvasitu}

\begin{document}

\date{}

\newcommand{\Addresses}{{
  \bigskip
  \footnotesize

  \textsc{Department of Mathematics, University at Buffalo}
  \par\nopagebreak
  \textsc{Buffalo, NY 14260-2900, USA}  
  \par\nopagebreak
  \textit{E-mail address}: \texttt{achirvas@buffalo.edu}


}}

\maketitle

\begin{abstract}
  For any square-summable commuting family $(A_i)_{i\in I}$ of complex $n\times n$ matrices there is a normal commuting family $(B_i)_i$ no farther from it, in squared normalized $\ell^2$ distance, than the diameter of the numerical range of $\sum_i A_i^* A_i$. Specializing in one direction (limiting case of the inequality for finite $I$) this recovers a result of M. Fraas: if $\sum_{i=1}^{\ell} A_i^* A_i$ is scalar for commuting $A_i\in M_n(\mathbb{C})$ then the $A_i$ are normal; specializing in another (singleton $I$) retrieves the well-known fact that close-to-isometric matrices are close to isometries. 
\end{abstract}


\vspace{.5cm}



\section*{Introduction}

The motivating result for this note is \cite[Theorem 1]{fraas2023commuting_xv1}: commuting $n\times n$ matrices $A_i$, $1\le i\le \ell$ with $\sum_i |A_i|^2=1$ are automatically normal. The ingenious proof in \cite{fraas2023commuting_xv1} relies on the decomposition theory (e.g. \cite[\S\S 4 and 5]{zbMATH06417267}) of {\it completely positive} \cite[Definition II.6.9.1]{blk} maps such as
\begin{equation*}
  M_n(\bC)=:M_n \ni
  X \xmapsto{\quad}\sum_{i=1}^{\ell}A_i XA_i^*
  \in M_n.
\end{equation*}
It seemed sensible, then, to seek for a more directly linear-algebraic proof and perhaps generalize the result in various ways in the process. For a positive integer $n$ write
\begin{equation*}
  \left\vvvert A = (a_{ij})_{i,j}\right\vvvert_2
  :=
  \frac 1{\sqrt n}\left(\sum_{i,j}|a_{ij}|^2\right)^{1/2}
\end{equation*}
for the normalized Hilbert-Schmidt norm on $M_n$ and similarly for tuples $(A_i)_{i\in I}$ of matrices:
\begin{equation*}
  \left\vvvert(A_i)_i\right\vvvert_2
  :=
  \left(\sum_{i\in I}\left\vvvert A_i\right\vvvert_2^2\right)^{1/2}.
\end{equation*}
With that in place, \Cref{th:onmatseq-rob} below is one possible generalization of \cite[Theorem 1]{fraas2023commuting_xv1}. 

\begin{theoremN}\label{thn:lesw}
  If $A_i\in M_n$, $i\in I$ commute then there are normal commuting $B_i\in M_n$ with
  \begin{equation}\label{eq:thn:lesw}
    \left\vvvert (A_i-B_i)_{i\in I}\right\vvvert_2^2
    \le
    \text{diameter of the numerical range of }
    \sum_i A_i^* A_i
  \end{equation}
  (said diameter counting as infinite if $\sum_i A_i^* A_i$ fails to converge).  \qedhere
\end{theoremN}
Recall that the {\it numerical range} \cite[Chapter 22]{hal_hspb_2e_1982} of an operator $A$ on a Hilbert space is 
\begin{equation}\label{eq:numrng}
  \bC\supset W(A):=\left\{\braket{\xi\mid A\xi}\ |\ \|\xi\|=1\right\},
\end{equation}
so that the right-hand side of \Cref{eq:thn:lesw} is a measure of how far $\sum A_i^* A_i$ is from being scalar. 

\Cref{th:red2qnilp} is yet another variation of the initial motivating result, retaining the normality context (as opposed to the {\it near}-normality of \Cref{thn:lesw}) but allowing for compact operators on infinite-dimensional Hilbert spaces:

\begin{theoremN}\label{thn:cpct}
  Let $A_i$ be commuting compact operators on a Hilbert space $H$ with $\sum A_i^* A_i$ strongly convergent to a scalar.
  
  There is then an orthogonal decomposition $H=H_{qn}\oplus H_{n}$, invariant under all $A_i$, such that the restrictions $A_i|_{H_n}$ are normal and $A_i|_{H_{qn}}$ are quasi-nilpotent.  \qedhere
\end{theoremN}

\subsection*{Acknowledgements}

This work is partially supported by NSF grant DMS-2001128.

I am grateful for stimulating suggestions and comments from M. Fraas, B. Passer and L. Paunescu. 


\section{Commuting operators square-summable to scalars}\label{se:ssumm}

We denote the {\it generalized $\lambda$-eigenspace} \cite[p.6-1]{hogben_hndbk-lalg_2e_2014} of an operator $A$ on a Hilbert space by
\begin{equation}\label{eq:geneigen}
  K_{\infty}(\lambda;A):=\bigcup_{n}K_{n}(\lambda;A)
  ,\quad
  K_{n}(\lambda;A) := \ker(\lambda-A)^n,
\end{equation}
with the subscript $1$ on the symbol $K_1(\lambda;A)$ for the plain eigenspace occasionally omitted. 



Recall \cite[Theorem 1]{fraas2023commuting_xv1}, stating that commuting matrices $A_i\in M_n$ such that 
\begin{equation}\label{eq:aiai1}
  \sum_i A_i^* A_i = 1
\end{equation}
are automatically normal. That result turns out to be robust under deformation in the appropriate sense: roughly speaking, commuting families $(A_i)_i$ of matrices that {\it almost} satisfy \Cref{eq:aiai1} are close to commuting normal families. To make sense of this, recall the {\it $L^p$-norm} $\|\cdot\|_p$ (\cite[\S I.8.7.3]{blk}, \cite[Definition XI.9.1]{ds_linop-2_1963}) defined on the ideal
\begin{equation*}
  \left\{\text{compact operators}\right\}=:
  \cK(H)
  \trianglelefteq
  \cL(H)
  :=
  \left\{\text{bounded operators on a Hilbert space }H\right\}
\end{equation*}

 of compact operators (if allowed to take infinite values):
\begin{equation*}
  \|T\|_p =
  \begin{cases}
    \left(\sum_{n\in \bZ_{\ge 0}}\mu_n(T)^p\right)^{1/p}& \text{for }1\le p<\infty\\
    \mu_0(T)=\text{usual operator norm }\|T\|&\text{for }p=\infty
  \end{cases}
\end{equation*}
where
\begin{equation*}
  \left(\mu_0(T)\ge \mu_1(T)\ge \cdots\right)
  :=
  \text{eigenvalues of }|T|:=(T^*T)^{1/2}
\end{equation*}
rearranged non-increasingly (the {\it characteristic numbers} \cite[\S XI.9]{ds_linop-2_1963} or {\it $s$-numbers} \cite[\S II.2]{gk_lin} of $T$). $\|\cdot\|_2$ is the familiar {\it Hilbert-Schmidt norm} \cite[Definition XI.6.1]{ds_linop-2_1963}. For an operator $A\in \cL(H)$ on a Hilbert space $H$ write
\begin{equation*}
  sw(A) := \diam W(A) = \sup_{z,z'\in W(A)}|z-z'|
  \quad
  \left(\text{{\it numerical spread} of $A$}\right)
\end{equation*}
($\braket{-\mid -}$ denoting the inner product in the ambient Hilbert space, linear in the second variable), where $W(A)$ is the numerical range \Cref{eq:numrng}. Note that $sw(A)$ vanishes precisely for scalar operators, so in general it is a measure of the discrepancy from being scalar.

\begin{remark}\label{re:spread}
  The numerical spread $sw(\cdot)$ is what most naturally fits the statement and proof of \Cref{th:onmatseq-rob}, but note that for {\it normal} operators it is nothing but the diameter of the (convex hull of the) spectrum \cite[Problem 216]{hal_hspb_2e_1982}.

  The term {\it spread} was in fact introduced for that quantity (diameter of the spectrum) in \cite[\S 1]{zbMATH03121640} in the context of matrices. As for links between the two notions of spread (numerical and plain, again for matrices), see e.g. \cite[\S 2]{zbMATH07333199}. 
\end{remark}

With this background, the statement alluded to above is 


\begin{theorem}\label{th:onmatseq-rob}
  For any commuting family $(A_i)\subset M_n$ with convergent $\sum_i A_i^* A_i$ there is a commuting family $(B_i)\subset M_n$ of normal matrices such that 
  \begin{equation}\label{eq:th:onmatseq-rob:normnorm}
    \sum_i \|A_i-B_i\|_2^2
    \le
    n \cdot sw\left(\sum_i A_i^* A_i\right).
  \end{equation}
\end{theorem}

Before going into the proof, note the following immediate consequence; it in turn recovers \cite[Theorem 1]{fraas2023commuting_xv1} by restricting attention to {\it finite} families. 

\begin{corollary}\label{cor:onmatseq}
  Commuting matrices $\{A_i\}_i\subset M_n$ with $\sum_i A^*_i A_i$ scalar are all normal and hence generate a commutative $C^*$-algebra.
\end{corollary}
\begin{proof}
  The first statement is an immediate consequence of \Cref{th:onmatseq-rob} (since the right-hand side of \Cref{eq:th:onmatseq-rob:normnorm} is now assumed to vanish). The second claim then follows from the {\it Putnam-Fuglede theorem} \cite[Problem 192]{hal_hspb_2e_1982}: commutation with a normal operator entails commutation with its adjoint.
\end{proof}

\begin{remark}\label{re:hyers-ulam}
  \Cref{th:onmatseq-rob} is an instance of {\it Hyers-Ulam(-Rassias) stability}: almost-linear operators between Banach spaces are close to linear operators \cite[Theorems 1.1 and 1.2]{jung_hur-stab}, surjective almost-isometries on Hilbert spaces are close to surjective isometries \cite[Theorem 13.4]{jung_hur-stab}, almost-homogeneous functions between Banach spaces are close to homogeneous functions \cite[Theorem 5.11]{jung_hur-stab}, etc.
\end{remark}

\pf{th:onmatseq-rob}
\begin{th:onmatseq-rob}
  Being commuting, the $A_i$ are simultaneously upper-triangular \cite[Theorem 2.3.3]{hj_mtrx} with respect to some orthonormal basis $(e_j)_1^n$. The sought-after $B_i$ will be the respective diagonals of the $A_i$:
  \begin{equation*}
    B_i:=
    \diag\left(\lambda_{i,j}:=\braket{e_j\mid A_i e_j},\ 1\le j\le n\right)
    ,\quad \forall i.
  \end{equation*}
  To verify \Cref{eq:th:onmatseq-rob:normnorm} set $A:=\sum_i A_i^* A_i$ and note first that 
  \begin{equation}\label{eq:aiej}
    \sum_i \|A_i e_j\|^2
    =    
    \sum_i \braket{e_j\mid A_i^* A_i e_j}
    =
    \braket{e_j\mid Te_j}
    ,\quad \forall 1\le j\le n.
  \end{equation}
  We claim next that for every $1\le j\le n$ we have
  \begin{equation}\label{eq:nontrivint}
    \bigcap_i K(\lambda_{i,j};A_i)\ne \{0\}.
  \end{equation}
  Momentarily taking this for granted, for each fixed $j$ we can simultaneously upper-triangularize the $A_i$ with respect to a new orthonormal basis with a vector $e'_j$ in \Cref{eq:nontrivint} listed first, so that 
  \begin{equation}\label{eq:biej}
    \sum_i \|B_i e_j\|^2
    =
    \sum_i \|A_i e'_j\|^2
    =
    \braket{e'_j\mid Te'_j}
    ,\quad \forall 1\le j\le n.
  \end{equation}
  \Cref{eq:aiej,eq:biej} are at most $sw(T)$ apart by the latter's definition, hence the conclusion upon summing over $1\le j\le n$. 
  It remains to settle \Cref{eq:nontrivint}. Since $A_i$ commute and thus preserve each other's eigenspaces, that assertion is equivalent to the non-trivial intersection of the {\it generalized} eigenspaces $K_{\infty}(\lambda_{i,j};A_i)$. Were that intersection trivial, the 1-dimensional module $A_i\mapsto \lambda_{i,j}$ of the commutative algebra $\cA$ generated by the $A_i$ would not appear as a subquotient in a {\it Jordan-H\"older filtration} \cite[Proposition III.3.7]{stenstr_quot} of either of the two $\cA$-modules
  \begin{equation*}
    K_{\infty}(\lambda_{i_0,j};A_{i_0})
    \quad\text{and}\quad
    V/K_{\infty}(\lambda_{i_0,j};A_{i_0}),\quad V:=\text{ambient space }\bC^n
  \end{equation*}
  for a fixed index $i_0$, so would not appear in such a filtration at all. This is at odds with the original triangularization with respect to $(e_j)$, hence the contradiction.
\end{th:onmatseq-rob}

\begin{remarks}\label{res:cholesky}
  \begin{enumerate}[(1),wide=0pt]
  \item The commutativity of the family $\{A_i\}$ of \Cref{cor:onmatseq} cannot be relaxed to simultaneous unitary upper-triangularizability (as the proof, appealing crucially to that commutativity, suggests): {\it every} positive operator on $\bC^n$ is expressible as $T^*T$ for upper-triangular $T$ (the celebrated {\it Cholesky factorization} \cite[Corollary 7.2.9]{hj_mtrx}), so it is enough to decompose $1\in M_n$ as a sum of non-diagonal positive operators, say
    \begin{equation*}
      1
      =
      \left(
        \begin{array}{rr}
          \frac 12 & \frac 12\\
          \frac 12 & \frac 12
        \end{array}
      \right)
      +
      \left(
        \begin{array}{rr}
          \frac 12 & -\frac 12\\
          -\frac 12 & \frac 12
        \end{array}
      \right),
    \end{equation*}
    and express each summand as $A_i^*A_i$ for upper-triangular $A_i$.

  \item The dependence on $n$ in \Cref{eq:th:onmatseq-rob:normnorm} vanishes upon substituting the {\it normalized} Hilbert-Schmidt norm $\frac {1}{\sqrt{n}}\|\cdot\|_2$ on $M_n$ for $\|\cdot\|_2$, as is customary in the literature on almost-commutative matrices (\cite[\S 1]{1002.3082v1}, \cite[\S 2]{MR4134896}, etc.).
  \end{enumerate}  
\end{remarks}

\cite[\S 4]{fraas2023commuting_xv1} observes that the {\it unilateral shift} \cite[Problem 82]{hal_hspb_2e_1982}
\begin{equation*}
  e_n\xmapsto{\quad S\quad} e_{n+1},\quad n\in \bZ_{\ge 0}
\end{equation*}
on a Hilbert space with orthonormal basis $(e_n)_{n\in \bZ_{\ge 0}}$ is a non-self-adjoint singleton giving a counterexample to \Cref{cor:onmatseq} in infinite-dimensional spaces. \Cref{cor:onmatseq} does, however, suggest a more hopeful infinite-dimensional variant: the $A_i\in \cL(H)$ might be commuting {\it compact} \cite[Definition I.8.1.1]{blk} operators on a Hilbert space $H$, with the convergence of \Cref{cor:onmatseq} valid in the {\it strong} \cite[Definition I.3.1.1]{blk} topology on $\cL(H)$. In that context, one result that requires little more than has already been noted is

\begin{theorem}\label{th:red2qnilp}
  Let $A_i\in \cK(H)$ be commuting compact operators on a Hilbert space with
  \begin{equation}\label{eq:sconv}
    \sum_i A_i^* A_i = 1
    \quad
    \left(\text{strong convergence}\right).
  \end{equation}
  There is then an orthogonal decomposition $H=H_{qn}\oplus H_{n}$, invariant under all $A_i$, such that the restrictions $A_i|_{H_n}$ are normal and $A_i|_{H_{qn}}$ are quasi-nilpotent. 
\end{theorem}

Recall (\cite[pre Theorem 1.5.2]{ringr_cpctnonsa}, \cite[\S V.3, post Theorem 3.5]{tl_fa_2e_1986}) that an operator on a Banach space is {\it quasi-nilpotent} if its spectrum is $\{0\}$ (equivalently: its {\it spectral radius} \cite[\S 1.5]{ringr_cpctnonsa} vanishes).

\begin{remark}\label{re:sameconv}
  Which of the six standard weaker-than-norm topologies \cite[\S I.3.1]{blk} on $\cL(H)$ (weak, $\sigma$-weak, strong, $\sigma$-strong, strong$^*$ and $\sigma$-strong$^*$) is explicitly mentioned in \Cref{eq:sconv} is a matter of taste: per \cite[Problem 120]{hal_hspb_2e_1982} (phrased in terms of plain sequences but applicable in the present generality), for bounded non-decreasing {\it nets} \cite[Definition 11.2]{wil_top} of positive operators those topologies induce the same notion of convergence.
\end{remark}

There is a theory of upper-triangularization for single compact operators: the term for what we would here call `upper-triangular' is {\it superdiagonal} in \cite[\S 4.3]{ringr_cpctnonsa}; other sources \cite[\S XI.10]{ds_linop-2_1963} work with {\it sub}diagonal operators instead. That material extends straightforwardly to commuting families of compact operators: the central result driving the theory, namely \cite[Theorem 4.2.1]{ringr_cpctnonsa} the fact that compact operators have non-trivial invariant subspaces, is now well-known (\cite[Theorem]{zbMATH03459430}, \cite{zbMATH03554933}) for the {\it commutant} of a non-zero compact operator. We take all of this for granted, along with the requisite background on compact-operator spectral theory.

Recall \cite[Theorem 1.8.1]{ringr_cpctnonsa}, in particular, that for compact $A\in \cL(H)$ the generalized eigenspaces \Cref{eq:geneigen} attached to $\lambda\in\bC^{\times}:=\bC\setminus\{0\}$ are finite-dimensional. In particular, the same goes for the (plain) eigenspaces $K(\lambda; A):=K_1(\lambda;A)$.

\pf{th:red2qnilp}
\begin{th:red2qnilp}
  We isolate a single operator $A:=A_{i_0}$ and fix a non-zero $\lambda\in\sigma(A)$. There is \cite[Theorem 1.8.1]{ringr_cpctnonsa} a direct-sum decomposition
  \begin{equation*}
    H=K_{\infty}(\lambda;A)\oplus R_{\infty}(\lambda;A)
  \end{equation*}
  (`$R$' for `range') where, by analogy to \Cref{eq:geneigen},
  \begin{equation*}
    R_{\infty}(\lambda;A)
    :=
    \bigcap_n R_{n}(\lambda;A)
    ,\quad
    R_{n}(\lambda;A) := \im(\lambda-A)^n = (\lambda-A)^n H.
  \end{equation*}
  In an appropriate orthonormal basis for $H$, compatible with the orthogonal decomposition $H=R_{\infty}(\lambda;A) \oplus R_{\infty}(\lambda;A)^{\perp}$, we have
  \begin{equation*}
    A=
    \begin{pmatrix}
      A' & \bullet\\
      0&T
    \end{pmatrix}
  \end{equation*}
  with $T$ finite (of width $\dim K_{\infty}(\lambda;A)<\infty$), upper triangular, with diagonal entries $\lambda$. The argument employed in the proof of \Cref{cor:onmatseq} will then show that the $\bullet$ block vanishes. This is sufficient to ensure that
  \begin{itemize}
  \item for non-zero $\lambda\in\sigma(A)$ the generalized eigenspaces are in fact eigenspaces:
    \begin{equation*}
      K_{\infty}(\lambda;A)
      =
      K(\lambda;A)
      ,\quad
      \forall \lambda\in
      \sigma(A)^{\times}
      :=
      \sigma(A)\setminus\{0\}; 
    \end{equation*}
  \item and those eigenspaces are mutually orthogonal for distinct $\lambda$:
    \begin{equation*}
      K(\lambda;A)\perp K(\lambda';A)
      ,\quad
      \forall \lambda\ne \lambda'\in \sigma(A)^{\times};
    \end{equation*}
  \item and finally, said eigenspaces are all orthogonal to the largest $A$-invariant subspace where $A$ is quasi-nilpotent:
    \begin{equation}\label{eq:rinfsigmax}
      K(\lambda;A)
      \perp
      R_{\infty}(\sigma(A)^{\times}; A)
      :=
      \bigcap_{\mu\in \sigma(A)^{\times}}
      R_{\infty}(\mu; A).
    \end{equation}
  \end{itemize}
  In short: $A$ is an orthogonal direct sum of a quasi-nilpotent compact operator and a normal compact operator, operating respectively on the space $R_{\infty}(\sigma(A)^{\times}; A)$ of \Cref{eq:rinfsigmax} and its orthogonal complement. Finally, setting
  \begin{equation*}
    H_{qn}:= \bigcap_i R_{\infty}(\sigma(A_i)^{\times}; A_i)
  \end{equation*}
  will do.
\end{th:red2qnilp}

\addcontentsline{toc}{section}{References}

\Addresses

\end{document}